\documentclass{amsart}
\usepackage[english]{babel}
\usepackage[utf8]{inputenc}
\usepackage{amsmath, amssymb,epic,graphicx,mathrsfs,enumerate}
\usepackage[all]{xy}
\usepackage{color}
\usepackage{comment}
\usepackage{enumitem}
\usepackage{hyperref}
\usepackage[]{frontespizio}

\usepackage{amsmath}
\usepackage{amsthm}
\usepackage{amssymb}
\usepackage{latexsym}
\usepackage{epsfig}

\DeclareMathOperator{\core}{Core} 
\DeclareMathOperator{\aut}{Aut} 
\DeclareMathOperator{\soc}{soc}

\DeclareMathOperator{\F}{F}
\DeclareMathOperator{\oo}{O}
\newcommand{\op}{\oo_p}

\DeclareMathOperator{\C}{C}
\DeclareMathOperator{\R}{R}
\DeclareMathOperator{\Z}{Z}

\newcommand{\gen}[1]{\left\langle#1\right\rangle} 
\newcommand{\al}{\alpha}

\newcommand{\st}{such that }
\newcommand{\ifa}{if and only if } 
\newcommand{\sg}{subgroup }

\newcommand{\gr}{group }

\newcommand{\Wh}{We have }

\newcommand{\ov}[1]{\overline{#1}}
\newcommand{\w}[1]{\widetilde{#1}}

\newcommand{\n}{\mathfrak{N}}
\newcommand{\ff}{\mathfrak{F}}
\newcommand{\fr}{\phi}
\newcommand{\G}{\Gamma}
\newcommand{\gff}{G^{\ff}}

\newcommand{\nn}{\mathrel{\unlhd}}
\newcommand{\sd}{\rtimes}

\newcommand{\gf}{\G_\ff}
\newcommand{\iso}{\mathcal{I}_\ff}
\newcommand{\is}[1]{\mathcal{I}_{#1}}

\newcommand{\gs}{\G_\mathfrak{S}}

\newcommand{\uu}{\mathfrak{U}}
\newcommand{\isou}{\mathcal{I}_\uu}
\newcommand{\dd}{\mathfrak{D}}
\newcommand{\isod}{\mathcal{I}_\dd}
\newcommand{\fp}{\overline{f(p)}}
\newcommand{\sss}{\mathfrak{S}}

\newcommand{\frf}{\fr_\ff}

\newtheorem{thm}{Theorem}

\newtheorem{step}{Step} \newtheorem{lemma}[thm]{Lemma}
\newtheorem{prop}[thm]{Proposition} 
 \newtheorem{defn}[thm]{Definition}

\newtheorem{stepp}{Step}
\numberwithin{equation}{section}

\renewcommand{\footnote}{\endnote}
\newcommand{\ignore}[1]{}\makeglossary

\begin{document}
	\bibliographystyle{amsplain}
\title[The non-$\mathfrak F$-graph of a finite group]{The non-$\mathfrak F$-graph of a finite group}

\author{Andrea Lucchini}
\address{Andrea Lucchini\\ Universit\`a di Padova\\  Dipartimento di Matematica \lq\lq Tullio Levi-Civita\rq\rq\\ Via Trieste 63, 35121 Padova, Italy\\email: lucchini@math.unipd.it}
\author{Daniele Nemmi}
\address{Daniele Nemmi\\ Universit\`a di Padova\\  Dipartimento di Matematica \lq\lq Tullio Levi-Civita\rq\rq\\ Via Trieste 63, 35121 Padova, Italy\\email: daniele.nemmi@studenti.unipd.it}


\begin{abstract} Given a formation $\mathfrak F$, we consider the graph  whose vertices are the elements of $G$ and where two  vertices $g,h\in G$  are adjacent \ifa $\gen{g,h}\notin\ff$. We are interested in the two following questions. Is the set of the isolated vertices of this graph a subgroup of $G?$ Is the subgraph obtained by deleting the isolated vertices a connected graph?	\end{abstract}
\maketitle

\hbox{}

\bibliographystyle{alpha}

\section{Introduction}
Let $\mathfrak F$ be a class of finite groups and $G$ a finite group. We may consider a graph $\w{\G}_\ff(G)$  whose vertices are the elements of $G$ and where two  vertices $g,h\in G$  are connected \ifa $\gen{g,h}\notin\ff$. We denote by $\iso(G)$ the set of isolated vertices of $\w{\G}_\ff(G)$. We define  the non-$\ff$ graph $\gf(G)$ of $G$ as the subgraph of $\w{\G}_\ff(G)$ obtained by deleting the isolated vertices. 
In the particular case when $\mathfrak F$ is the class $\mathfrak A$ of the abelian groups, the graph $\Gamma_{\mathfrak A}(G)$ has been introduced by Erd\"{o}s and it is known with the name of non-commuting graph (see for example\cite{ncg}, \cite{neu}). If $\mathfrak F$ is the class $\mathfrak N$ of the finite nilpotent groups, then $\Gamma_{\mathfrak N}(G)$ is the non-nilpotent graph, studied for example in \cite{az}. When $\mathfrak F$ is the class $\mathfrak S$ of the finite soluble groups, we obtain the non-soluble graph (see \cite{ns}).

\

A group (resp. subgroup) is called an $\ff$-\emph{group} (resp. $\ff$-\emph{subgroup}) if it belongs to $\ff$. We say that $\mathfrak F$ is  \emph{hereditary} whenever if $G\in\ff$ and $H\leq G$, then $H\in\ff$. If $\mathfrak F$ is hereditary, it is interesting to consider the intersection $\fr_\ff(G)$ of all maximal $\mathfrak F$-subgroups of $G,$ that is, the subgroups which are maximal with respect to being an $\mathfrak F$-group.
It turns out that if $\mathfrak F \in \{\mathfrak A, \mathfrak N,
\mathfrak S\},$ then $\fr_\ff(G)=\iso(G)$ for any finite group $G.$ Indeed $\mathcal{I}_{\mathfrak A}(G)=Z(G),$ $\mathcal{I}_{\mathfrak N}(G)=Z_\infty(G)$ \cite[Proposition 2.1]{az}, $\mathcal{I}_{\mathfrak S}(G)=\R(G)$
\cite[Theorem 1.1]{gu}, denoting by $Z_\infty(G)$ and $\R(G)$, respectively, the hypercenter and the soluble radical of $G.$
This motivates the following definition: we say that    $\ff$ is \emph{regular} if $\ff$ is hereditary and  $\frf(G)=\iso(G)$ for every finite group $G$.

\

The first question that we address in the paper is how to characterize the hereditary saturated formations that are regular. Recall that a formation $\mathfrak F$ is a class of groups which is closed under taking homomorphic images and subdirect products. The second condition ensures the existence of the
$\mathfrak F$-residual $G^{\mathfrak F}$ of each group $G,$ that is, the smallest normal subgroup of $G$ whose factor
group is in $\mathfrak F$. A formation $\mathfrak F$ is said to be saturated if $G \in\mathfrak F$ whenever the Frattini factor
$G/\Phi(G)$ is in $\mathfrak F.$ A \gr $G$ is \emph{critical} for $\ff$ (or $\ff$-\emph{critical}) if $G\notin\ff$ and every proper \sg of $G$ lies in $\ff$, while a group $G$ is \emph{strongly critical} for $\ff$ if $G\notin\ff$ and every proper subgroup and proper quotient of $G$ lies in $\ff$.

\begin{thm}\label{regolari}Let $\mathfrak F$ be an hereditary saturated formation, with $\mathfrak A \subseteq \ff \subseteq \mathfrak S.$ Then $\ff$ is regular if and only if every finite group $G$ which	is soluble and strongly critical for $\ff$ has the property that $G/\soc(G)$ is cyclic.
\end{thm}

It follows from Theorem \ref{regolari} that a formation is not in general regular. For example, if $\mathfrak U$ is the formation of the finite supersoluble groups, then there exists a  strongly critical group $G$ for $\mathfrak U$ such that $\soc(G)$ is an elementary abelian group of order 25 and $G/\soc(G)$ is isomorphic to the quaternion group $Q_8.$ It is an interesting question to see if and when $\iso(G)$ is a subgroup of $G$. 

\

Consider the class $\ff$  of finite groups in which normality is transitive. The group $G=:\langle a,b,c \mid a^5=1, b^5=1, c^4=1, [a,b]=1, a^c=a^2, b^c=b^3\rangle$ 
 is critical for $\ff$ (see \cite{transitiveN}). Then $\gen{a,g}$ and $\gen{b,g}$ are proper subgroups for every $g\in G$, so they belong to the class, while $\gen{ab,y}=G$ does not belong to the class. Thus $a, b \in \iso(G)$ but $ab\notin \iso(G)$. So in general $\iso(G)$ is not a subgroup of $G.$

\

We say that a formation $\mathfrak F$ is \emph{semiregular} if $\iso(G)\leq G$ for any finite group $G.$ In Section \ref{riduzione} we will investigate the structure of a group $G$ which is  minimal with respect to the property that  $\iso(G)$ is not a subgroup. To state our result we need to recall another definition: we say that $\ff$ is \emph{2-recognizable} whenever a group $G$ belongs to $\ff$
if all $2$-generated subgroups of $G$ belong to $\ff.$

\begin{thm}\label{semi}
	Let $\mathfrak F$ be an hereditary saturated formation, with $\mathfrak A \subseteq \ff \subseteq \mathfrak S.$
	Assume that $\ff$ is 2-recognizable and not semiregular and let $G$ be a finite group of  minimal order with respect to the property that  $\iso(G)$ is not a  \sg of $G.$ Then $G$ is a primitive monolithic soluble group. Moreover, if $N=\soc(G)$ and $S$ is a complement of $N$ in $G,$ then the following hold.
	\begin{enumerate}
		\item $N=\soc(G)=\gff$. 
		\item  $N\!\gen{s}\in\ff$ for every $s\in S$; in particular $S$ is not cyclic. 
		\item if $n\in N$ and $s\in S,$ then  $ns\in\iso(G)$ \ifa $N\!\gen{s,t}\in\ff$ for all $t\in S$; in particular  $ns\in\iso(G)$ \ifa $s\in\iso(G)$.
		\item Suppose that $\ff$ is locally defined by the formation function $f$ and, for every prime $p,$ let 
		$\fp$ be the formation of the finite groups $X$ with the property that $X/\op(X)\in f(p).$
		If $K\leq S$, we have that $NK\in\ff$ \ifa $K\in\fp$, in particular $\iso(G)=N\is{\fp}(S)$, where $p$ is the unique prime dividing $|N|$.
	\end{enumerate}
\end{thm}
As an application of the previous theorem we will prove.
\begin{thm}\label{appli}
The following formations are semiregular:
\begin{enumerate}
	\item the formation $\mathfrak U$ of the finite supersoluble groups.
	\item  the formation $\mathfrak D=\mathfrak N\mathfrak A$ of the finite groups with nilpotent derived subgroup.
	\item the formation $\n^t$ of the finite groups with Fitting length less or equal then $t,$ for any $t\in \mathbb N.$
 	\item the formation $\sss_p\n^t$ of the finite groups $G$ with $G/\op(G) \in \n^t.$
\end{enumerate}
\end{thm}
We will say that a formation $\ff$ is \emph{connected} if 
the graph $\gf(G)$ is connected for any finite group $G.$ 
In Section \ref{connesso} we consider the case when
$\ff$ is a 2-recognisable hereditary saturated semiregular formation with $\mathfrak A \subseteq \ff \subseteq \mathfrak S$. In particular we investigate the structure of a group $G$ of minimal order with the property that $\gf(G)$ is not connected (when $\ff$ is not connected) and we use this information to prove the following result.

\begin{thm}\label{conreg}
		Let $\mathfrak F$ be an hereditary saturated formation, with $\mathfrak A \subseteq \ff \subseteq \mathfrak S.$ If $\ff$ is regular, then $\ff$ is connected.
\end{thm}
A corollary of this result is \cite[Theorem 5.1]{az}, stating that the non-nilpotent graph $\Gamma_{\mathfrak N}(G)$ is connected for any finite group $G$. Moreover our approach allows to prove:

\begin{thm}
	\label{conaltri}
		If $\ff\in\{\uu,\dd,\sss_p\n^t,\n^t\}$, then $\ff$ is connected. 
\end{thm}

Recall that a graph is said to be embeddable in the plane, or \emph{planar}, if it can be drawn in the plane so that its edges intersect only at their ends. Abdollahi and Zarrin proved
that if $G$ is a finite non-nilpotent group, then the non-nilpotent graph $\Gamma_{\mathfrak N}(G)$ is planar if and only if $G\cong S_3$ (see \cite[Theorem 6.1]{az}). We generalize this result proving:
\begin{thm}\label{planar}
	Let $\ff$ be a 2-recognizable, hereditary, semiregular formation, with $\n\subseteq \ff,$ and let $G$ be a finite group.  Then $\gf(G)$ is planar if and only if either $G\in \ff$ or $G\cong S_3$.
\end{thm}
\section{Some preliminary results}
This section contains some auxiliary results, that will be needed in our proofs.
\begin{defn} Let $G$ be a finite group. We denote by $V(G)$ the subset of $G$ consisting of the elements $x$ with the property that $G=\langle x, y\rangle$ for some $y.$
\end{defn}

\begin{prop}\label{lemma10}
	Let $G$ be  a primitive monolithic soluble group. Let $N=\soc(G)$ and $H$ a core-free maximal subgroup of $G$. Given $1\neq  h\in H$ and $n\in N$, $hn \in V(G)$ if and only if $h\in V(H).$
\end{prop}
\begin{proof} Clearly if $hn\in V(G),$ then $h\in V(H).$ Conversely assume that $h\in V(H)$ and let $n\in N.$ There exists $k\in H$ such that $\langle h, k\rangle=H.$ For any $m\in N,$ let $H_m:=\langle hn, km\rangle.$ Since $H_mN=\langle h, k \rangle N=G,$ either $H_m=G$ or $H_m$ is a complement of $N$ in $G.$ In particular, if we assume, by contradiction, $hn\notin V(G),$ then $H_m$ is a complement of $N$ in $G$ for any $m\in G$, and consequently $H_m=H^{g_m}$ for some $g_m\in G.$ If $H_{m_1}=H_{m_2}$ then $m_1^{-1}m_2=(km_1)^{-1}(km_2)\in H_{m_1}\cap N=1$ so $m_2=m_1.$  Since $N_G(H)=H,$ $H$ has precisely $|G:H|=|N|$ conjugates in $G$ and therefore $\{H_m\mid m\in N\}$ is the set of all the conjugates in $G$. This implies $1\neq hn \in \bigcap_{g\in G}H^g=\core_G(H)=1,$ a contradiction.
\end{proof}

\begin{lemma}\label{N}Let $\ff$ be a saturated formation with $\ff \subseteq \mathfrak S$ and let $G$ be a finite group. Suppose $G\notin\ff$  but every proper quotient is in $\ff$. Then either $\R(G)=1$ or $G$ is a primitive monolithic soluble group and $\soc(G)=\gff$. 
\end{lemma}
\begin{proof}
	If $\R(G)\neq1$, we have $G/\R(G)\in\ff$, hence $G/\R(G)$ is soluble, which implies that $G$ is soluble. If $G$ contains two different minimal normal subgroups, $N_1$ and $N_2,$ then $G=G/(N_1\cap N_2)\leq G/N_1\times G/N_2 \in \ff,$ against our assumption. So $\soc(G)$ is the unique minimal normal subgroup of $G$. Moreover $G/\soc(G)\in \ff$, hence $\soc(G)=\gff.$ Finally, since $\ff$ is a saturated formation and $G\notin \ff$, it must be $\phi(G)=1,$ so $G$ is a primitive monolithic soluble group.
\end{proof}
The following is immediate.
\begin{lemma}\label{FFF} Let $g,h\in G$ and $N\nn G$.
	\begin{enumerate}[label=(\alph*)]
		\item If $gN$ and $hN$ are adjacent vertices of $\gf(G/N)$, then $g$ and $h$ are adjacent vertices of $\gf(G)$.
		\item If $g\in \iso(G),$ then $gN\in \iso(G/N).$
		\item $\iso(G)^\sigma=\iso(G)$ for every $\sigma\in\aut(G)$.
	\end{enumerate}
\end{lemma}

\begin{prop}\cite[Theorem A]{sk}\label{frf}
	Let $\ff$ be a saturated formation. Let $H\leq G$ and $N\nn G$. Then
	\begin{enumerate}[label=(\alph*)]
		\item If $H\in\ff$, then $H\frf(G)\in\ff$;
		\item If $N\nn\frf(G)$, then $\frf(G)/N=\frf(G/N)$.
	\end{enumerate}
\end{prop}

\section{Proof of Theorem \ref{regolari}}

Let $\mathfrak F$ be an hereditary saturated formation, with $\mathfrak A \subseteq \ff \subseteq \mathfrak S.$ 

\

First we claim that
$\frf(G)\subseteq \iso(G)$.
Since $\ff$ contains all the cyclic groups, by  Proposition \ref{frf} (a), $\gen{x}\frf(G)\in\ff$ for any $x\in G.$ The conclusion follows from the fact that $\ff$ is hereditary.

\

Suppose that $\ff$ is regular and let $G$ be a soluble strongly critical \gr for $\ff$. By Lemma \ref{N}, $G$ is a primitive monolithic soluble group. Moreover, since $G$ is critical for $\ff$,
all the maximal subgroups of $G$ are in $\ff,$ and therefore 
$\iso(G)=\frf(G)=\fr(G)=1.
$ Let $N=\soc(G)$ and $S$ a complement of $N$ in $G.$ Fix $1\neq n \in N.$ Since $n\notin \iso(G)$, $\langle n, g\rangle \notin \ff$ for some $g\in G.$ Since $G$ is $\ff$-critical, it must be $\langle n, g\rangle=G$ and therefore $G/N$ is cyclic.

\

Conversely, suppose that $\ff$ is not regular and every soluble strongly critical group $G$ for $\ff$ is \st $G/\soc(G)$ is cyclic. Let $G$ be a smallest finite group \st $\frf(G)\subset\iso(G)$. Of course $G\notin\ff$, otherwise $G=\frf(G)=\iso(G)$. Let $x\in\iso(G)\setminus\frf(G)$ and  let $H$ be an $\ff$-maximal \sg of $G$ which does not contain $x$. 
\begin{step}
$G=\gen{x,H}$.
\end{step}
\begin{proof}
Suppose, by contradiction, $\gen{x,H}<G$. Then  $x\in\iso(\gen{x,H})=\frf(\gen{x,H})$, hence, by   Proposition \ref{frf} (a),  $\gen{x,H}=\frf(\gen{x,H})H\in\ff$, against the fact that $H$ is an $\ff$-maximal subgroup of $G$.
\end{proof}
\begin{step}\label{regolari:quo}
If $1\neq M\nn G$, then $G/M\in\ff$.
\end{step}
\begin{proof}
	By Lemma~\ref{FFF} and the minimality of $G,$  $xM\in\iso(G/M)=\frf(G/M),$ hence $G/M=\gen{xM,HM/M}=\frf(G/M)HM/M\in\ff$, since $HM/M\cong H/(M\cap H)\in\ff$.
\end{proof}
\begin{step}
$G$ is a primitive monolithic soluble group and $\soc(G)=\gff$.
\end{step}
\begin{proof}
From Step \ref{regolari:quo} we are in the hypotheses of Lemma \ref{N}. If $\R(G)=1$, by \cite[Theorem 6.4]{gu}, for every $1\neq g_1\in G$  there exist $g_2\in G$ \st $\gen{g_1,g_2}$ is not soluble, and then $\gen{g_1,g_2}\notin\ff$, since $\ff$ contains only soluble groups. So, $\iso(G)=1$, hence $\frf(G)=1$, which means $\iso(G)=\frf(G)$, against the assumptions on $G$.
\end{proof}

Let $N=\soc(G)$, $S$ a complement of $N$ in $G$ and write $x=\bar n\bar s$ with $\bar n\in N, \bar s\in S.$
\begin{step}
There exists $1\neq n^*\in N\cap\iso(G)$.
\end{step}
\begin{proof} We may assume $\bar s\neq 1$ (otherwise $x=\bar n\in N\cap \iso(G)$) and $\bar s\notin V(S)$ (otherwise, by Proposition \ref{lemma10}, $\langle x, g\rangle=G\notin \ff$ for some $g\in G$ and $x=\bar n \bar s\notin \iso(G)).$ Since $C_G(N)=N,$ there exists $m\in N$ such that $x^m\neq x.$ We claim that $n^*=[m,x^{-1}]\in N\cap \iso(G).$ Indeed let $g\in G.$ Since $\bar s\not\in V(S),$ $K:=\gen{x,x^m,g}=
	\gen{x, n^*x, g}=\gen{\bar n\bar s,n^*\bar n\bar s,g}\leq N\gen{\bar s,g}<G.$
In particular, again by the minimality of $G,$ $x,x^m\in\iso(K)=\frf(K)$, hence $K=\frf(K)\gen{g}$ and, since $\gen{g}\in\ff$, $K\in\ff$. Since $\langle n^*, g\rangle \leq K,$ we conclude $\langle n^*, g\rangle \in \ff.$
\end{proof}

\begin{step}
$S$ is not cyclic.
\end{step}
\begin{proof}
	Suppose, by contradiction, $S=\gen{s}.$ Since $N$ is an irreducible $S$-module and $n^*\neq 1,$ we have $\gen{n^*,s}=G.$ However $n^*\in \iso(G),$ so this would imply $G\in \ff.$
\end{proof}

\begin{step}
$N\subseteq \iso(G)$.
\end{step}
\begin{proof}
Suppose, by contradiction, that there exist $m\in N$ and $g\in G$ such that
$\gen{g,m}\not\in \ff.$ This implies $K:=N\gen{g}\notin \ff.$ By the previous step, $K<G.$ By Lemma \ref{FFF}, $(n^*)^s\in \iso(G)$ for any $s\in S.$ So in particular
$X=\{(n^*)^s\mid s\in S\}\subseteq \iso(K).$ However, by the minimality of $G,$ $\iso(K)=\frf(K)$ is a subgroup of $G$, so $\langle X\rangle =  N\leq \frf(K)$ and consequently $K=\frf(K)\langle g \rangle\in \ff.$
	\end{proof}
\begin{step}
$G$ is a strongly critical \gr for $\ff$.
\end{step}
\begin{proof}
By Step \ref{regolari:quo}, we just need to prove that every maximal \sg of $G$ is in $\ff$. Notice that $S\cong G/N\in\ff$, and so does every conjugate of $S$. The other maximal subgroups of $G$ are of the form $K:=NM$, with $M$ maximal in $S$. In particular, by the minimality of $G$, $\iso(K)=\frf(K)$, and, by the previous step, $N\leq\frf(K)$. Hence $K=\frf(K)M\in\ff$, since $M\in\ff$.
\end{proof}
Finally, $G$ is a soluble strongly critical \gr for $\ff$, so $G/N\cong S$ is cyclic, but we excluded this possibility in Step 5. \Wh a contradiction, so $\ff$ must be regular.

\section{Proof of Theorem \ref{semi}}\label{riduzione}

To prove the theorem we need the following lemma.

\begin{lemma}\label{Icyc}
Suppose that $\ff$ is a 2-recognizable formation. If $\iso(G)$ is a subgroup of $G$ and $G=\iso(G)\!\gen{g}$ for some $g\in G$, then $G\in\ff$.
\end{lemma}
\begin{proof}
Let $x$ be an arbitrary element of $G$. We have $x=ig^\al$ for some $i\in\iso(G)$ and $\al\in\mathbb{N}$. Moreover $\gen{g,ig^\al}=\gen{g,i}\in\ff$, since $i\in\iso(G)$. Hence $g\in\iso(G)$, so $G=\iso(G)$ and, because $\ff$ is 2-recognizable, $G\in\ff$.
\end{proof}
\begin{proof}[Proof of the Theorem \ref{semi}]
 Let $x,y\in\iso(G)$ \st $xy\notin\iso(G)$. There exists  $g\in G$  \st $\gen{xy, g}\notin \ff$. Notice that the minimality property of $G$ implies $G=\gen{x,y,g}.$ Let $M$ be a non-trivial normal subgroup of $G$ and set $I/M:=\iso(G/M)\nn G/M$. By Lemma \ref{FFF}, $xM,yM\in\iso(G/M)$. Since $G=\gen{x,y,g}$, we have $\gen{gM}I/M=G/M$. 
By Lemma \ref{Icyc}, $G/M\in\ff$.
So we are in the hypotheses of Lemma \ref{N}. If $\R(G)=1$, then,  as in the proof of Theorem \ref{regolari}, $\iso(G)=1$, in contradiction with the assumption that $\iso(G)$ is not a subgroup of $G.$ So $G$ is a primitive monolithic soluble group and $N=\soc(G)=\gff.$ 

We will show now that there is an element $1\neq n^*\in N\cap \iso(G)$. We write  $x$ in the form $x=\bar{n}\bar{s},$ with $\bar n\in N$ and $\bar s\in S$. If $\bar s=1,$ then $x\in N\cap \iso(G)$ and we are done (notice that $xy \not\in \iso(G)$ implies $x\neq 1).$
Suppose $\bar s\neq1$. Since $G\notin \ff,$ $x\notin V(G)$, hence $\bar s\notin V(S)$ by Proposition \ref{lemma10}. Since $C_N(x)\neq N,$ there exists $m\in N$ such that $x^m \neq x.$ We claim that $n^*:=[m,x^{-1}]\in N\cap \iso(G).$ Indeed let $g\in G.$ Since $\bar s\not\in V(S),$ $K:=\gen{x,x^m,g}=
\gen{x, n^*x, g}=\gen{\bar n\bar s,n^*\bar n\bar s,g}\leq N\gen{\bar s,g}<G.$
In particular $x,x^m\in\iso(K)$ and $K=\iso(K)\gen{g}$ and therefore  $K\in\ff$ by Lemma \ref{Icyc}.

We prove now that $N\subseteq\iso(G)$. As in Step 6 of the proof of Theorem \ref{regolari}, assume by contradiction that $\gen{g,m}\notin \ff,$ for some $m\in N$ and $g\in G.$ Setting $K:=N\gen{g},$ it follows, with the same argument, that  $N\leq \iso(K)$ and consequently $K=\iso(K)\gen{g}\in \ff$  by Lemma \ref{Icyc}, a contradiction.

Let $s$ be an arbitrary element of $S$ and let 
$H:=N\!\gen{s}.$ Since $N\subseteq \iso(G)\cap H\subseteq \iso(H),$ we deduce that $H\in \ff$ from Lemma \ref{Icyc}. This proves (2).

Let now $n\in N$ and $s\in S$. If $ns\in\iso(G)$, 
then $ns\notin V(G)$ and therefore $s\notin V(S)$ by 
 Proposition \ref{lemma10}.  Let $t$  be an arbitrary element of $S$ and set $H:=N\!\gen{s,t}<G$. Since $H<G$, by the minimality of $G,$ $\iso(H)$ is a subgroup of $G,$ and therefore $N\!\gen{s}\leq \iso(H)$, and consequently  $H=\iso(H)\!\gen{t}$ and $H\in\ff$ by Lemma \ref{Icyc}. If, on the contrary, $ns\notin\iso(G)$, then there exist $n^*\in N$ and $s^*\in S$ \st $\gen{n^*s^*,ns}\notin \ff$, hence $N\!\gen{s,s^*}\notin\ff$. This proves (3).
 
Finally, we prove (4). Let $K\leq S$. Suppose $H:=NK\in\ff$. Let $U/V$ be a $p$-chief factor of $H$ with $U\leq N$. Since $H\in\ff$, we have $\aut_H(U/V)=H/C_H(U/V)\in f(p)$; moreover, since $N$ is abelian, $N\leq\C_H(U/V)$, so $\C_H(U/V)=N\C_K(U/V)$ and hence $\aut_K(U/V)\cong\aut_H(U/V)\in f(p)$. Let $1=N_0\nn N_1\nn\dots\nn N_t=N$ with $N_i/N_{i-1}$ a chief factor of $H$ for every $i$. Since $N$ is a $p$-group, $\aut_K(N_i/N_{i-1})\in f(p)$ for every $i$, so $K/T\in f(p)$ with $T:=\bigcap_{i=1}^t C_K(N_i/N_{i-1})$. Since $C_T(N)\leq C_K(N)=1,$ $T$ is a $p$-group, hence $K^{f(p)}\leq T\leq \op(K)$ and  $K\in\fp$. Conversely,
suppose $K\in\fp$. Let $1=N_0\nn\dots\nn N_t=N=NK_0\nn\dots\nn NK_s=NK=H$ be a chief series of $H$ and denote by $\F(H)$ the Fitting subgroup of $H.$ If $1\leq i \leq t,$ then $\aut_H(N_i/N_{i-1})$ is an epimorphic image of $H/\F(H)$, since $\F(H)\leq C_H(N_i/N_{i-1})$. On the other hand, $\F(H)=N\op(K)$, hence $H/\F(H)\cong K/\op(K)\in f(p)$, and so $\aut_H(N_i/N_{i-1})\in f(p)$. Consider now $\aut_H(NK_j/NK_{j-1})$ for $1\leq j\leq a$ and let $q$ be the prime dividing $|NK_j/NK_{j-1}|$. Then we have $H/C_H(NK_j/NK_{j-1})\cong K/C_K(K_j/K_{j-1})=\aut_K(K_j/K_{j-1})\in f(q)$, since $NK_j/NK_{j-1}\cong K_j/K_{j-1}$ is a chief factor of $K$ and $K\in\ff$. So $H$ satisfies all the local conditions, and then it is in $\ff$.
\end{proof}

\section{Proof of Theorem \ref{appli}}

\begin{prop}\label{ssol}
The formation $\uu$ of finite supersoluble groups is semiregular.
\end{prop}
\begin{proof}The formation
$\uu$ is 2-recognizable since every $\uu$-critical \gr is 2-generated (see for instance \cite[Example 1]{minnonF}). Assume by contradiction that $\uu$ is not semiregular and let $G$ be a group of minimal order with respect to the property that $\isou(G)$ is not a subgroup. We can apply Theorem \ref{semi}. Let $N=\soc(G)$: we have $|N|=p^k$ for a prime $p$ and some $k$. Let $q\neq p$ be another prime divisor of the order of a complement $S$ of $N$ in $G$ and  choose $s\in S$ with $|s|=q$.  By Theorem \ref{semi}, $N\!\gen{s}\in\uu.$ Applying Maschke's Theorem, $N$ can be decomposed into a direct sum of irreducible submodules and, since $N\!\gen{s}$ is supersoluble, these submodules must have order $p$. So $s$ acts faithfully on a cyclic group of order $p$, hence  $q$ divides $p-1$ and in particular $q<p$. If $p\mid |S|$, then $p$ would be the greatest prime divisor of $|S|$. Since $S\in\uu$, the Sylow $p$-\sg of $S$ is normal in $S$. However, since $S$ acts faithfully and irreducibly on the finite $p$-group $N$,
$\oo_p(S)=1$. This implies  $\gcd(|N|,|S|)=1$ and since $N\gen{s}\in\uu$ for every $s\in S$, the exponent of $S$ divides $p-1$. The local definition $f(p)$ of $\uu$ is the formation of abelian group with exponent dividing $p-1$, therefore, since $p$ does not divide $|S|$, $NK\in\uu$ \ifa $K$ is abelian, hence $\isou(G)=N\Z(S)$ is a subgroup of $G$, so we reached a contradiction.
\end{proof}

\begin{prop}
The formation $\dd$ of the finite groups with nilpotent derived subgroup is semiregular.
\end{prop}
\begin{proof}
The $\dd$-critical groups are $2$-generated (see for instance \cite[Example 2]{minnonF}), so $\dd$ is 2-recognizable.
Suppose by contradiction it is not semiregular and let $G$ be a minimal example of group \st $\isod(G)$ is not a subgroup. We can apply Theorem \ref{semi}. Let $N=\soc(G)$ and $S$ a complement of $N$. We will prove that if $H\leq S$, then $NH\in\dd$ \ifa $H$ is abelian. Since $\dd$ has local screen $f$ with $f(q)$ the formation of the abelian groups for every prime $q$, if $H$ is abelian, then $NH\in\dd$. On the other hand, suppose $NH\in\dd$. Let $1=N_0\nn\dots\nn N_l=N$ be a composition series of $N$ as $H$-module. Let $V_i:=N_i/N_{i-1}$ and $C_i:=\C_H(V_i)$. For every $1\leq i\leq l,$ we have that $H/C_i\cong \aut_{NH}(V_i)$ is abelian, since $V_i$ is a chief factor of a group in $\dd$. Then we have that $H/T$ is abelian, with $T:=\bigcap_{i=1}^lC_i$. Therefore  $H'\leq T$. 
Since $C_T(N)\leq C_H(N)=1,$ $T$ is a $p$-group,  but $|S'|$ is not divisible by $p$ (otherwise, since $S^\prime$ is nilpotent,
we would have $\op(S)\neq 1),$
 so $H^\prime \leq T\cap S^\prime=1$ and $H$ is abelian. Hence $\isod(G)=N\Z(S)$, a contradiction.
\end{proof}

Let $\n^t$ the formations of finite groups with Fitting length less or equal then $t$. It is a 2-recognizable, saturated formation \cite[Example 3]{minnonF}. As an immediate application of Theorem \ref{semi}, we prove its semiregularity by proving that the formation $\fp=\sss_p\n^{t-1}$ is semiregular for every prime $p$.
We will need two preliminary lemmas.
\begin{lemma}\label{snt}
$\sss_p\n$ is regular for every prime $p$.
\end{lemma}
\begin{proof}
Let $G=N\sd S$ be a strongly-critical group for $\sss_p\n$. The socle  $N$ of $G$ is a $q$-group. If $q=p$, then, since $S\cong G/N\in\sss_p\n$ and $\op(S)=1$, if follows $S\in\n$ and $G\in\sss_p\n$, so it must be $q\neq p$. If $K<S$, then $NK\in\sss_p\n$.  Since $C_S(N)=1,$  we deduce $\op(NK)=1$, hence $NK\in\n$, which implies that $NK$ is a $q$-group (otherwise $C_K(N)\neq 1)$. \Wh then that all proper subgroups of $S$ are $q$-groups, but $S$ itself is not a $q$-group, so $S$ must be cyclic of order a prime $r\neq q$. We deduce from Theorem \ref{regolari} that $\sss_p\n$ is  regular.
\end{proof}

\begin{lemma}
$\sss_p\n^t$ is a 2-recognizable saturated formation for every $t$ and every prime $p$.
\end{lemma}
\begin{proof}The formation $\sss_p\n^t$ is saturated (see \cite[IV, 3.13 and 4.8]{dh}).
 We prove by induction on $t$ that  $\sss_p\n^t$ is a 2-recognizable. We have seen in Lemma \ref{snt} that $\sss_p\n$ is 2-recognizable for every prime $p$. Let $t\neq 1$ and let $G$ be a group of minimal order with respect to the property that every 2-generated subgroup of $G$ is in
	$\sss_p\n^t$ but $G$ is not. Clearly $G$ is strongly critical for $\sss_p\n^t$, so, by Lemma \ref{N}, $G=N\sd S$, where $N=\soc(G)$ is an elementary abelian group of prime power order and $S\in\sss_p\n^t$. If $N$ is a $p$-group, then $G\in\sss_p\n^t$, hence $N$ is a $q$-group with $q\neq p$.  If $K<S$, then $NK\in \sss_p\n^t.$ Since $C_K(N)=1,$ it must be $\oo_p(NK)=1$ so $NK\in\n^t$. Moreover the Fitting subgroup $\F(NK)$ of $NK$ coincides with $\oo_q(NK)=N\oo_q(K)$ and therefore $K\in\sss_q\n^{t-1}$, so $S$ is critical for $\sss_q\n^{t-1}$. Since, by induction,
	 $\sss_q\n^{t-1}$ is 2-recognizable, the group $S$ is 2-generated. By Proposition \ref{lemma10}, $G$ itself is $2$-generated and hence $G\in\sss_p\n^t$, a contradiction.
\end{proof}

\begin{prop}
$\sss_p\n^t$ is semiregular for every $t$ and every prime $p$.
\end{prop}
\begin{proof}
We  prove by induction on $t$ that $\sss_p\n^t$ is semiregular for every $t$. By Lemma \ref{snt} we may assume $t>1.$ Suppose by contradiction that $\sss_p\n^t$ is not semiregular and let $G$ be a minimal example of group \st $\is{\sss_p\n^t}(G)$ is not a subgroup. We can apply Theorem \ref{semi}. Let $N=\soc(G)$ and $S$ a complement of $N$. Since $S\in\sss_p\n^t$, if $N$ were a $p$-group, then $G$ would be in $\sss_p\n^t$, hence $N$ is a $q$-group with $q\neq p$. Let now $s,t\in S$ and $K:=\gen{s,t}$: since $\F(NK)=N\oo_q(K)$, we have $NK\in\sss_p\n^t$ if and only if $NK\in\n^t$, if and only if $K\in\sss_q\n^{t-1}$. Hence by induction we conclude that $\is{\sss_p\n^t}(G)=N\is{\sss_q\n^{t-1}}(S)$ is a subgroup, a contradiction.
\end{proof}

\begin{prop}
$\n^t$ is semiregular for every $t$.
\end{prop}
\begin{proof}
Since $\fp=\sss_p\n^{t-1}$, the statement follows from Theorem \ref{semi} and Proposition \ref{snt}.
\end{proof}

\section{Connectedness of $\gf$}\label{connesso}

In this section we study for which formations the graph $\gf(G)$ is connected for every finite group $G$. In the spirit of the previous sections we will build, under the additional assumption that $\ff$ is semiregular, a smallest \gr $G$  \st $\gf(G)$ is not connected. First we need a preliminary lemma.
\begin{lemma}\label{swa}
Let $G$ be a 2-generated finite soluble group, with $G\notin \ff.$ If $x,y\in V(G),$ then $x$ and  $y$ belong to the same connected component of $\gf(G).$
\begin{proof}
Consider the graph $\Delta(G)$ whose vertices are the elements of
$V(G)$ and in which $g_1, g_2$ are adjacent if and only if $\langle g_1, g_2 \rangle=G.$ If $G$ is soluble then $\Delta(G)$ is a connected graph (see \cite[Theorem 1]{cl}). The conclusion follows from the fact that $\Delta(G)$ is a subgraph of $\gf(G).$
\end{proof}
\end{lemma}
\begin{thm}\label{connection}
		Let $\mathfrak F$ be a 2-recognizable, hereditary, saturated formation, with $\mathfrak A \subseteq \ff \subseteq \mathfrak S.$
Assume that $\ff$ is semiregular and suppose that there exists a finite group $G$ such that $\gf(G)$ is not connected. If $G$ has minimal order with respect to this property, then $G$ is a primitive monolithic soluble group, $N=\soc(G)=\gff$ and $N\subseteq\iso(G)$. Moreover, the same statements of point (2-4) of Theorem \ref{semi} hold. With the same notation, we have also that $\G_{\fp}(S)$ is not connected.
\end{thm}
\noindent Given a finite group $X$, we will write $x_1\sim x_2$ to denote that $x_1$ and $x_2$ are two adjacent vertices of $\gf(X)$
	and $x_1\approx x_2$ if $x_1$ and $x_2$ belong to the same connected component of $\gf(X).$
We divide the proof in the following steps.

\begin{stepp}
$G$ is a primitive monolithic soluble group and
 $N=\soc(G)=\gff$.
\end{stepp}
\begin{proof} Suppose there exists $1\neq M\nn G$ \st $G/M\notin\ff$. Set $I/M:=\iso(G/M)\nn G/M$ and let $a_1M,a_2M\notin I/M$. \Wh $a_1M\approx a_2M$ by minimality of $G$. Since, by Lemma \ref{FFF} (a), $g_1M\sim g_2M$ implies $g_1\sim g_2$, we can ``lift'' a path from $a_1M$ to $a_2M$ in $\gf(G/M)$ to a path from $a_1$ to $a_2$ in $\gf(G)$, so $a_1\approx a_2$. So there exists a unique connected component of $\gf(G),$ say $\Omega,$ containing $G\setminus I.$ 
 If $I\in\ff$, then every element of $I\setminus \iso(G)$ must be adjacent to an element of $G\setminus I,$ so $I\setminus \iso(G) \subseteq \Omega$. But this implies $\Omega=G\setminus \iso(G)$, and consequently $\gf(G)$ is connected. Therefore  $I\notin\ff$. Since $\ff$ is 2-recognizable, this implies $\iso(I)<I.$ Let $H$ be a maximal subgroup of $G$ containing $I$. Since
 $\gf(H)$ is connected, there exists a unique connected component of
 $\gf(G),$ say $\Delta,$ containing $H\setminus \iso(H).$ Of course $I\setminus \iso(I) \subseteq H\setminus \iso(H),$ so $I\setminus \iso(I)\subseteq \Delta.$ Recall that $G\setminus I\subseteq \Omega.$ Moreover if $x\in \iso(I) \setminus \iso(G),$ then $x\sim y$ for some $y\in G\setminus I,$ so $\iso(I) \setminus \iso(G)\subseteq \Omega$. If $\Delta\cap \Omega\neq \emptyset,$ then $\Delta=\Omega=G\setminus \iso(G)$ and $\gf(G)$ is connected. So we may assume $\Delta\cap \Omega=\emptyset,$ and consequently
 $(H\setminus \iso(H))\cap (H\setminus I)=\emptyset$, i.e. $H=I\cup \iso(H).$ Since $H\notin \ff$ and $\ff$ is 2-recognizable, $\iso(H)\neq H,$ and consequently $H=I.$ If $g\in G\setminus I,$ then $G=\gen{g}I$, so $G/M=\gen{gM}I/M=\gen{gM}\iso(G/M)$ and, by Lemma \ref{Icyc}, $G/M\in\ff$, a contradiction.
So all the proper factors of $G$ are in $\ff$ and we may   use Lemma \ref{N}. If $\R(G)=1$, then $\iso(G)=1$. Let $a,b\in G,$ both different from $1$. By \cite[Theorem 6.4]{gu}  there is a path in $\gs(G)$ from $a$ to $b$. This path is also a path in $\gf(G)$ since $H\notin\mathfrak{S}$ implies $H\notin\ff$ for every \gr $H$. So if $\R(G)=1$, then $\gf(G)$ is connected. Hence we conclude that
 $G$ is a primitive monolithic soluble group and $N=\soc(G)=\gff$.
\end{proof}

\begin{stepp}
$N\subseteq\iso(G)$.
\end{stepp}
\begin{proof} Since $\iso(G)\nn G$ and $N$ is the unique minimal normal subgroup, if $\iso(G)\neq 1$, then $N\subseteq\iso(G)$. Hence we may assume by contradiction that $\iso(G)=1$. 
Let $S$ be a complement of $N$ in $G.$  Suppose that $S=\gen{s}$ is cyclic.
Since $S$ is a maximal subgroup of $G$, $\langle g,s\rangle =G$ for any $g\notin \gen{s},$ hence there exists a connected component $\Lambda$ of $\gf(G)$ containing $s$ and $G\setminus \gen{s}.$
Moreover, every nontrivial element of $S$, being non-isolated in $\gf(G),$ is adjacent to some element of $G\setminus S$, so $\Lambda=G\setminus \{1\}$ and $\gf(G)$ is connected, a contradiction. So we may assume that $S$ is not cyclic. Take now $n_1, n_2\in N\setminus \{1\}$ and for $i\in\{1,2\}$ let $M_i<S$ \st $n_i\notin\iso(NM_i)$ (this is  possible since $S$ is not cyclic). \Wh $N_i:=N\cap\iso(NM_i)<N$, so $N_1\cup N_2\neq {N}$ and there exists  $n\in N\setminus(N_1\cup N_2)$. We have then $n_1	\approx n$ in $\gf(NM_1)$ and $n_2\approx n$ in $\gf(NM_2)$, therefore $n_1\approx n_2$ in $\gf(G)$. Hence there exists a connected component $\Pi$ of $\gf(G)$ containing $N\setminus\{1\}$. Let now $g=ns$ be an arbitrary element of $G\setminus N.$ First assume $g\notin V(G)$. Since $N\not\leq C_G(g),$ there exists  $n^*\in N\setminus \{n\}$ with the property that $g=(n^*s)^x$ for some $x\in G.$
 We claim that $g\in \Pi.$ Since $n^*n^{-1}\neq 1$, there exists $\bar g= \bar n \bar s$ such that  $\bar g\sim n^*n^{-1}$. Set $H:=N\!\gen{s,\ov{s}}$ (it is a proper \sg of $G$, since for Proposition \ref{lemma10}, $s\notin V(S))$. If $g\notin\iso(H)$, then  $ns\approx n^*n^{-1}$ (since $\gf(H)$ is connected) and then $g\in \Pi.$ Assume   $g\in\iso(H)$.
We have  $n^*s\not\in\iso(H)$, (otherwise, since $\iso(H)$ is a subgroup, $(n^*s)(ns)^{-1}=n^*n^{-1}\in \iso(H)$), but then 
 $n^*s\approx n^*n^{-1}$ in $\gf(H)$ and consequently $n^*s\in \Pi.$   This implies $g=(n^*s)^x\in \Pi^x=\Pi$ (notice that $\Pi^x=\Pi$ since $N\setminus \{1\}
 \in \Pi \cap \Pi^x).$ Suppose now $g\in V(G)$. 
 Choose $n_1, n_2\in N$ and $t \in S$ such that $n_2\sim n_1t$ and 
 set $H:=N\gen{s,t}.$
 If $H=G$, then $t\in V(S)$ and consequently $n_1t\in V(G).$ Since $G$ is soluble, it follows from Lemma \ref{swa} that $g\approx n_1t \approx n_2$ and $g\in \Pi.$
 If  $H<G$, then $ms \notin  \iso(H)$ for some $m\in N$ (otherwise $N\leq \iso(H)).$ By Proposition \ref{lemma10}, $ms\in V(G)$ and, again by Lemma \ref{swa},
 $g \approx ms.$ Moreover, since $\gf(H)$ is connected, $ms\approx n_2$. So $g \approx n_2$ and therefore $g\in \Pi.$
 We reached in this way the conclusion  that $\gf(G)$ is connected, against the assumptions on $G$.
\end{proof}

\begin{stepp}
Statements (2-4) of Theorem \ref{semi} hold.
\end{stepp}
\begin{proof} We can use the same argument of the proof of Theorem \ref{semi}.
\end{proof}

\begin{stepp}
$\G_{\fp}(S)$ is not connected.
\end{stepp}

\begin{proof}
Suppose that $\G_{\fp}(S)$  is connected. Let $s,t\in S$ \st $s\sim t$ in $\G_{\fp}(S)$. We claim that $ns\approx mt$ for every $n,m\in N$. Suppose $\gen{s,t}=S$. By Proposition \ref{lemma10} $ns,mt\in V(G)$  so, by Lemma \ref{swa}, they are in the same connected component of $\gf(G)$. Suppose instead that $\gen{s,t}<S$. We have that $H:=N\gen{s,t}<G$ is not in $\ff$ since $\gen{s,t}\notin\fp$. Therefore $ns$ and $mt$ are not isolated in $H$ and, for minimality, $\gf(H)$ is connected, so $ns\approx mt$ in $\gf(G)$ too. Choose now two non-isolated vertices $n_1s_1,n_2s_2\in\gf(G)$ with $n_1,n_2\in N$ and $s_1,s_2\in S$. Since they are not isolated, $s_1,s_2\notin\is{\fp}(S)$, hence there is a path $s_1=z_0\sim\dots\sim z_l=s_2$ in $\G_{\fp}(S)$ and since $z_i\sim z_{i+1}$, we have, for every $m,h\in N$ and every $i$, that $mz_i\approx hz_{i+1}$ in $\gf(G)$ and so $n_1s_1\approx n_2s_2$, a contradiction.
\end{proof}

\begin{proof}[Proof of Theorem \ref{conreg}]
Suppose $G$ has minimal order with respect to the property that $\gf(G)$ is not connected. By Theorem \ref{connection}, $G$ is a primitive monolithic group and $N\nn\iso(G)=\frf(G)$. By Proposition \ref{frf}, $\frf(G)/N=\frf(G/N)=G/N$, hence $\frf(G)=G$, a contradiction.
\end{proof}

\begin{proof}[Proof of Theorem \ref{conaltri}]
It follows applying Theorem \ref{connection}, noticing that:
\begin{itemize}
\item If $\ff\in\{\uu,\dd\}$, then $\G_{\fp}(S)=\G_{\mathfrak{A}}(S)$  is connected. 
\item If $\ff=\sss_p\n^t$ for some prime $p$ and some $t$, then $\G_{\fp}(S)=\G_{\sss_q\n^{t-1}}(S)$ for some other prime $q$. Therefore we can use induction on $t$, considering that $\sss_p\n$ is regular for every $p$ and that Theorem \ref{conreg} holds.
\item If $\ff=\n^t$ for some $t$, then $\G_{\fp}(S)=\G_{\sss_p\n^{t-1}}(S)$ for some prime $p$ and we can use the point above.\qedhere
\end{itemize} 
\end{proof}
\section{Planarity of $\gf$}
The generating graph $\tilde \Delta(G)$ of a finite group $G$ is the graph whose vertices are the elements of $G$ and in which two vertices $g_1$ and $g_2$ are adjacent if and only if $\langle g_1, g_2\rangle =G.$ Moreover $\Delta(G)$ is the subgraph of $\tilde \Delta(G)$ induced by the subset of its non isolated vertices. Notice that if $G$ is a 2-generated $\ff$-critical group, then $\gf(G)\cong  \Delta(G).$ 

\begin{proof}[Proof of Theorem \ref{planar}]
One implication is easy: if $G\in \ff$ then $\gf(G)$ is a null graph, while if $G\cong S_3$ and $S_3\notin \mathfrak F,$ then $\gf(G)\cong \Delta(G)$ is planar, as it is noticed in \cite{planar}. Conversely, suppose $G\notin \ff$ and $\gf(G)$ is planar. Since $\ff$ is 2-recognizable, there exist $a, b \in G$ such that $\langle a, b \rangle \notin \ff.$ Since $\Delta(\langle a, b\rangle)$ is a subgraph of $\gf(G)$, it must be planar. Finite groups with planar generating graph have been completely classified in \cite{planar}. In particular, if $\Delta(X)$ is planar, then either $X$ is nilpotent or $X\in \{S_3, D_6\}.$ Since $\n \subseteq \ff,$ $\langle a, b \rangle$ is not nilpotent, so either $\langle a, b \rangle \cong S_3$ or $\langle a, b \rangle \cong D_6.$ Since $D_6\cong S_3\times C_2$ and $C_2\in \ff,$ $D_6\notin \ff$ implies $S_3\notin \ff.$ Let $A$ be the set of the non-central involutions of $D_6$ and let $B$ the set of the elements of $D_6$ of order divisible by 3: then $\Gamma_\ff(D_6)$
contains the complete bipartite graph whose partition has the parts $A$ and $B$, so it is not planar.
Hence $\gen{a,b}$ can only be isomorphic to $S_3$. We show that all the elements of $G$ have order less or equal to $3$. Suppose in fact that there is $g\in G$ \st $|g|\geq 4$. Since $\gf(G)$ is planar, $g\notin \iso(G)$  would imply that it generates a copy of $S_3$ with another element, but this is impossible since $|g|\geq 4$. \Wh then that $g\in \iso(G)$ and therefore $|\iso(G)|\geq 4.$ We claim that this is not possible. Indeed $G$ contains $X=\langle a, b \rangle \cong S_3\notin \mathfrak F$. Since $\ff$ is semiregular,  $I:=\iso(G)$ is a  normal subgroup of $G$. Since $I\cap X=1,$ for every $x,y \in I$ we have
$$\frac{\langle ax, by \rangle}{I\cap \langle ax, by \rangle} \cong  \frac{\langle ax, by \rangle I}{I} \cong \frac{\langle a, b \rangle I}{I}\cong {\langle a, b \rangle }\cong S_3\notin \ff,
$$ hence
$\langle ax, by \rangle \notin \mathfrak F.$ But then $\gf(G)$ contains the complete bipartite graph on the two parts $aI$ and $bI$
and then it is not planar. We have so proved that all the elements of $G$ have order order less or equal than $3$. 
Groups with this property have been classified in \cite{groups23}.
Since $G$ is not nilpotent and contains a subgroup isomorphic to 
$S_3$, $G\cong A\rtimes \langle x \rangle$, with $A\cong C_3^t$ and $x$ acting on $A$ sending every element into its
inverse. In particular the subgraph of $\gf(G)$ induced by the $3^t$ involutions is complete, so it is planar only if $t=1$, i.e. $G\cong S_3.$
\end{proof}

\end{document}

\begin{thm}\label{regolari}
	A formation $\ff$ is regular \ifa every finite \gr $G$ which is soluble and strongly critical for $\ff$ has the property that $G/\soc(G)$ is cyclic.
\end{thm}